\documentclass[12pt]{amsart}

\usepackage{amsmath,amssymb,amscd,amsfonts,amsthm,verbatim}
\usepackage[all,cmtip]{xy}
\usepackage{paralist}
\usepackage[mathscr]{eucal}
\usepackage{color}
\usepackage{pdfsync}
\usepackage[]{fontenc}
\usepackage{enumerate}

\usepackage[pdftex]{hyperref}
\hypersetup{citecolor=blue,linktocpage}

\usepackage[colorinlistoftodos]{todonotes}

\newcommand\rp[1]{\todo[color=yellow]{#1}}

\newtheorem{thm}{Theorem}[section]
\newtheorem{lem}[thm]{Lemma}

\newtheorem{assu-nota}[thm]{Assumption--Notation}

\theoremstyle{definition}
\newtheorem{defn}[thm]{Definition}
\newtheorem{rem}[thm]{Remark}

\newcommand{\Q}{\mathbb Q}

\newcommand{\pp}{\mathbb P}

\newcommand{\bP}{\mathbb P}

\newcommand{\OO}{\mathcal O}

\DeclareMathOperator{\Pic}{Pic}

\DeclareMathOperator{\Alb}{Alb}
\DeclareMathOperator{\alb}{alb}

\DeclareMathOperator{\Spec}{Spec}

\DeclareMathOperator{\kod}{kod}

\newcommand{\ol}{\overline}

\numberwithin{equation}{section}

\title{Surfaces on the Severi line}
\author{Miguel Angel Barja,  Rita Pardini and Lidia Stoppino}

\begin{document}
\begin{abstract} Let $S$ be a minimal complex surface of general type and
of maximal Albanese dimension; by the Severi inequality one has $K^2_S\ge 4\chi(\OO_S)$. We prove that the equality $K^2_S=4\chi(\OO_S)$ holds if and only if $q(S):= h^1(\OO_S)=2$ and the canonical model of $S$ is a double cover of the Albanese surface branched on an ample divisor with at most negligible singularities.\par
\end{abstract}
\maketitle
\tableofcontents

\section{Introduction}

Let $S$ be a  complex surface; we write $\chi(S)$ for $\chi(\OO_S)=\chi(\omega_S)$ and $q(S):=h^1(\OO_S)$ for the irregularity.

 Recall that $S$ has  maximal Albanese dimension if its Albanese map is generically finite; if  $S$ is minimal and of maximal Albanese dimension, then the numerical invariants of $S$ satisfy the so-called Severi inequality:
\begin{equation}\label{eq:severi}
K_S^2\ge 4\chi(S).
\end{equation}
In this paper we characterize the minimal surfaces of general type  and maximal Albanese dimension on the Severi line, namely those  for which the inequality \eqref{eq:severi} is actually an equality:
\begin{thm}\label{main}
Let $S$ be a minimal surface of general type of maximal Albanese dimension.
The equality $K_S^2=4\chi(S)$ holds if and only if:
\begin{itemize}
\item[(a)]
  $q(S)=2$;  and
 \item[(b)]  the canonical model of $S$ is a double cover of the Albanese surface $\Alb(S)$ whose  branch divisor is ample and has at most negligible singularities.
 \end{itemize}
\end{thm}
We point out that the assumption that $S$ be of general type in  Theorem \ref{main} cannot be weakened: if $E$ is a curve of genus 1 and $C$ is a curve of genus $g>1$, then $Y=E\times C$ has maximal Albanese dimension and lies on the Severi line but $q(Y)=g+1$ can be arbitrarily large. More generally, it is easy to see that a properly elliptic surface $S$ of maximal Albanese dimension is isotrivial  and has $\chi=0$; using the results of \cite{serrano} and  \cite{serranoC}, it follows that $S$ is a free quotient of a product $E\times C$ as above.

In order to put this result in perspective, we give a  brief account of  the history of the Severi inequality and its generalizations. The inequality, claimed erroneously by Severi (\cite{sev}) in 1932   and then posed as a conjecture in  the 1970's by Catanese and Reid (\cite{cat}, \cite{miles}), was first  proven by Manetti  (\cite{man}) at the end of the 1990's under the additional assumption that $K_S$ be ample, and then   a full proof was given by the second named author in 2004 (\cite{Pa}).

The proof  given in  \cite{man} yields  as a byproduct
the  characterization of   the surfaces  on the Severi line with $K_S$ ample: these  are   double covers of an abelian surface branched on a smooth ample divisor.
The  statement of Theorem \ref{main} is the natural extension of this characterization to the general case,  and as such has been  widely believed to be true (cf. \cite[\S 5.2]{ML-P}), but  for a long time  it could not be proved.
For instance the ``first'' case, $K^2=4$ and  $\chi=1$,  corresponding  to a double cover of a principally polarized  abelian surface branched on a divisor of $|2\Theta|$,  has been the object of two papers: the characterization has  been proven in \cite{cm}  under the assumption that the bicanonical map be non-birational and then  in \cite{cmp} without this assumption. Both proofs are quite technical and rely on very ad-hoc arguments, that do not extend to higher values of $\chi$. For $\chi>1$  up to now it was not even known  whether the surfaces on the Severi line have bounded irregularity $q(S)$.

Since the full proof of the Severi inequality given in \cite{Pa} consists in applying the so-called covering trick (cf. \S \ref{covering-trick}) and then taking  a limit, it is not possible to extract from it  information on the case where equality holds.
 The breakthrough  that allows us to overcome this difficulty  is the work \cite{barja-severi} by the first named author, where  the Severi inequality  is reinterpreted as an inequality involving  the self-intersection of the nef line bundle $K_S$ and  its    continuous rank (cf. \S \ref{severi-inequalities}), and extended to the much more general situation  of an arbitrary  nef line bundle on a $n$-dimensional variety  of maximal Albanese dimension (Clifford-Severi inequality). By combining  this more general form of the inequality with the covering trick,  we  are able to avoid the limit process and give a fairly  quick proof of Theorem \ref{main}. Another essential ingredient of the proof is the classical Castelnuovo's bound on the geometric genus of a curve of degree $d$ in $\pp^r$.

\smallskip

\noindent{\bf Notation and conventions:} We work over the complex numbers.
All varieties are assumed to be projective.
Given a surface   $S$ with $q(S)>0$,  we denote by $\Alb(S)$ its Albanese variety and by  $\alb_S\colon S\to \Alb(S)$ the Albanese map.\\
We will use interchangeably the notion of line bundles and (Cartier) divisors, but adopt only the additive notation.
We use $\equiv$ to denote numerical  equivalence of ($\Q$)-divisors.
We say that a divisor $D$ is {\it pseudoeffective} if $D\cdot A\geq 0$
for any nef divisor  $A$.

\smallskip

\noindent {\bf Acknowledgements:} We would like to thank S. Rollenske for enlightening conversations on this result.
Once this paper was finished, the authors were informed by Lu and Zuo of the existence of \cite{LZ}, where Theorem \ref{main} is also proved, with different techniques. 

\section{Set-up and preliminary results}

In this section we describe the construction and preliminary results needed for the proof of the main theorem.
Most of the  theory (e.g.,  the covering trick and the continuous rank) can be developed  for varieties of arbitrary dimension, but we will stick to the surface case because this  is what we need. For more general treatments the reader can consult  the  given references.

\subsection{Covering trick}\label{covering-trick}
In this subsection we recall  the construction introduced  in \cite{Pa} in order to prove the Severi inequality for surfaces of maximal Albanese dimension.
We shall repeatedly use it throughout the paper.
\medskip

We find it convenient to introduce some terminology for maps to abelian varieties:
\begin{defn}
 Let $a\colon X\longrightarrow A$ be a morphism from a projective variety $X$ to an abelian variety $A$.

 We say that the map $a$ is {\em generating} if $a(X)$ generates $A$, and that it is  {\em strongly generating} if the map $a^*\colon \Pic^0(A)\to \Pic^0(X)$ is injective.

We say that $X$ is of {\em maximal $a$-dimension} if the image of $X$ via $a$ in $A$ has dimension equal to $\dim X$; if $a$ is the Albanese map then we say that
$X$   has {\em maximal Albanese dimension}.
\end{defn}
\begin{rem} Notice that, as suggested by the terminology, a strongly generating map  $a\colon X\to A$ is generating. Indeed,  if  $Z$ is the abelian subvariety of $A$ generated by $a(X)$, then the kernel  of
$$a^*\colon \Pic^0(A)\to \Pic^0(X)$$
contains the kernel of the natural surjection  $\Pic^0(A)\to \Pic^0(Z)$;  therefore  $\ker a^*=\{0\}$ implies that  $\Pic^0(A)\to \Pic^0(Z)$ is an isomorphism, that is, that  $A=Z$.

Notice also  that, if
$a\colon X\longrightarrow A$  is a generating map, then  there is a factorization $a=\phi\circ a'$ where $A'$ is the abelian variety dual to $\Pic^0(A)/\ker a^*$, the  map $ a'\colon X\to A'$  is strongly generating and $f\colon A'\to A$ is  the  dual isogeny of $\Pic^0(A)\to \Pic^0(A)/\ker a^*$.

Observe that the Albanese map of a smooth variety $X$  is strongly generating and  the  map $a\colon X\to A$ coincides with  the Albanese map if and only if it is strongly generating and $\dim A=q(X)$.
\end{rem}
Fix a   smooth projective  surface $S$  and a strongly generating morphism  $a\colon S\rightarrow A$  to an abelian variety $A$.

The covering trick goes as follows.
Fix an integer $d>0$, and consider the morphism $\mu_d\colon A \longrightarrow A$ given by multiplication by $d$.
Consider the surface $S_d$ defined by the cartesian diagram:
$$
\xymatrix{
S_d  \ar[r]^{p} \ar[d]_{a_d}& S \ar^{a}[d] \\
A\ar^{\mu_d}[r] &A\\}
$$

\begin{lem} \label{lem:Ctrick}
The surface $S_d$ is smooth and connected. In addition:
\begin{enumerate}
\item the morphism $a_d$ is strongly generating;
\item  the $a_d$-dimension of $S_d$ is equal to the $a$-dimension of $S$;
\item assume that $S$ has maximal $a$-dimension; then  $S_d$ is minimal if and only if $S$ is.
\end{enumerate}
\end{lem}
\begin{proof} The surface $S_d$ is smooth since $p$ is \'etale. The connectedness of $S_d$ follows from the property of $a$ of being strongly generating, as follows.
 The surface $S_d$ is smooth, hence its number of connected components is equal to $h^0(\OO_{S_d})$.  By construction, the surface $S_d$ is isomorphic to  $\Spec(\oplus_{\eta\in A[d]}  a^*\eta) $, hence   $h^0(\OO_{S_d})=\sum _{\eta \in A[d]}h^0(a^*\eta)=1$, since $a^*\eta$ is a non-trivial torsion line bundle for every $0\ne \eta\in A[d]$ by the injectivity of $a^*$.

(i) Since $a$ is strongly generating, the kernel of $$(a\circ p)^*\colon \Pic^0(A)\longrightarrow \Pic^0(S_d)$$ is equal to $\Pic^0(A)[d]$. From  the commutativity of the diagram that defines $S_d$ and  the fact that the kernel of $\mu_d^*\colon \Pic^0(A)\to\Pic^0(A)$ is also equal to $\Pic^0(A)[d]$, it follows that $\ker a_d^*=\{0\}$.
\smallskip

(ii)  The claim follows again from the commutativity of the diagram, since $\mu_d$ is a finite map.
\smallskip

(iii)  Both $S$ and $S_d$ have non-negative Kodaira dimension since they have maximal $a$-dimension.
In addition, one has $K_{S_d}=p^*K_S$, because $p$ is \'etale, and therefore $K_{S_d}$ is nef  iff $K_S$ is.
\end{proof}
Let $H$ be a fixed very ample divisor on $A$.
Set  $C:=a^*(H)$ and $C_d:=a_d^*(H)$.
By \cite[Prop 2.3.5]{LB}, we have that
\begin{equation}\label{cd}
\mu_d^*H\equiv d^2H,
\quad
\mbox{ and so }
\quad
C_d\equiv \frac{1}{d^2} p^*C
\end{equation}
and
\begin{equation}\label{cdK}
C_d^2=d^{2q-4} C^2, \quad C_d\cdot  K_{S_d}=d^{2q-2} C\cdot  K_S.
\end{equation}
The map $p$ is \'etale of degree $d^{2q}$, so $K_{S_d}=p^*K_S$ and  the following numerical  invariants are multiplied by this degree:
$$K_{S_d}^2=d^{2q}K_S^2\qquad \chi(S_d)=d^{2q}\chi(S).$$

So,  given $a\colon S\longrightarrow  A$ as above and  a very ample divisor $H$ on $A$, for  any $d>0$  the covering trick
produces the surface $S_d$,  the map $a_d\colon S_d\longrightarrow A$, and the divisor $C_d$.

The behavior of the invariants  for $d\gg0$  (and in particular the divergence between the canonical divisor and the divisor $C_d$) is a key point in the covering trick.

A first application is the following nefness result.

\begin{lem}\label{K-Cnef}
With the above notation, if $S$ is minimal  of general type, then for $d\gg0$ the divisor $K_{S_d}-C_d$ is nef.
\end{lem}
\begin{proof}
Let us consider  the canonical model $S_{\rm can}$ of $S$ and the natural map
$$S \stackrel{\nu}{\longrightarrow} S_{\rm can}, $$
which is the contraction of the configurations of   $(-2)$-curves  of  $S$.
The map $a\colon S\to A$ contracts the rational curves and therefore   it induces a map $\overline  a\colon S_{\rm can}\to A$;
we set   $\overline C:={\ol a}^*C$, so that $C=\nu^*\ol C$.

The canonical divisor $K_{S_{\rm can}}$ is ample, hence there is $\epsilon>0$ such that the $\Q$-divisor $K_{S_{\rm can}}-t\overline C$ is ample for every rational number $t$ with  $|t|<\epsilon$. Pulling back to  the smooth surface $S$, we obtain that   $$\nu^*(K_{S_{\rm can}}-t\overline C)=K_S-tC$$  is a nef divisor.

By \eqref{cd} and \eqref{cdK},  we have
$$K_{S_d}-C_d\equiv p^*\left(K_S-\frac{1}{d^2}C\right).$$
So for any $d$ such that $\frac{1}{d^2}<\epsilon $  we have that $K_{S_d}-C_d$ is nef.
\end{proof}

\subsection{Continuous rank and Severi-type inequalities}\label{severi-inequalities}

Let $S$ be a smooth  surface with a generating map $a\colon S\longrightarrow A$ to  an abelian variety $A$.
Let $L$ be a line bundle on $S$.
For any non-negative integer $i$ the {\em  $i$-th continuous rank } (\cite{barja-severi}) of $L$ is the integer
$$
h_a^i(S,L):=\min \left\{h^i(S,L+ a^*\alpha) \mid \alpha \in \Pic^0(A)\right\}.
$$
If no confusion is likely to arise we shall write $L+ \alpha$ for $L+ a^*\alpha$.

\begin{rem}
Observe that the  $i$-th  continuous rank of $L$ is just the rank of the $i$-th Fourier-Mukai transform of $L$ with respect to the map $a$.
If $L = K_S + D$ with $D$ nef, and $S$ is of maximal $a$-dimension, then $h^0_a(S,L) = \chi(S, L)$ by the Generic Vanishing theorem  (Theorem B in \cite{pp}, cf. \cite{gl} for the case $D=0$).
In particular, $h^0_{\alb_S}(S, K_S)=\chi(S)$.
\end{rem}

In \cite{barja-severi} the following Severi-type inequalities  are proved (see \cite{barja-severi} for a much more complete statement).

\begin{thm}[Barja, \cite{barja-severi}]\label{teomab}
Let $S$ be a smooth surface with a generating morphism $a\colon S\rightarrow A$ to  an abelian variety $A$;
suppose that $S$ is of maximal $a$-dimension  and $L$ is a nef divisor on $S$. Then
\begin{itemize}
\item [(i)] The following inequality holds:
$$
L^2\geq 2h^0_a(S,L).
$$
\item [(ii)] If, moreover, $K_S-L$ is pseudoeffective, then
$$
L^2\geq 4h^0_a(S,L).
$$

\item [(iii)] If $K_S\equiv L_1+L_2$ with $L_i$ nef, then
$$
K_S^2\geq 4\chi(S)+4h^1_{\alb_S}(S,L_1).
$$
\end{itemize}
\end{thm}
We now state some consequences of these results for the extremal cases of the above inequalities.

\begin{rem}\label{K-decomposable}
The proof of Theorem \ref{teomab}(iii) in \cite[Corollary D (Cor. 4.9)]{barja-severi} shows that if we have a nef decomposition $K_S\equiv L_1+L_2$, then  $$K^2_S-4\chi(S)\geq (L_1^2-4h_{\alb_S}^0(S,L_1))+(L_2^2-4h_{\alb_S}^0(S,L_2))+4h_{\alb_S}^1(S,L_1).$$ Hence it follows that if  we have
a minimal surface of maximal Albanese dimension  $S$ with such a  decomposition $K_S\equiv L_1+L_2$ and such that  $K_S^2=4\chi(S)$, then the following holds:
\begin{enumerate}
\item $L_i^2=4h_{\alb_S}^0(S,L_i) $, for $i=1,2$;
\item $h_{\alb_S}^1(S,L_i)=0$ , for $i=1,2$.
\end{enumerate}
\end{rem}

\medskip

\subsection{Applications of Castelnuovo's bound}\label{appl-castelnuovo}

We now see that the classical Castelnuovo bound combined with the above results implies a condition on the asymptotic behavior of the invariants $C_d^2$ and $h^0_a(S_d,C_d)$ arising from the covering trick (\S \ref{covering-trick}).
Let us recall the result of Castelnuovo (\cite[III.2]{ACGH}).
\begin{thm}[Castelnuovo's bound]\label{castelnuovo}
Let $C$ be a smooth curve that admits a birational $g_d^r$. Then the genus of $C$ satisfies the inequality
$$
g\leq \frac{m(m-1)}{2}(r-1)+m\epsilon,
$$
where the integers $m$ and $\epsilon$ are respectively the quotient and the remainder  of the division of $d-1$  by $r-1$.
\end{thm}
\begin{lem}\label{birational}
Let $S$ be a surface of Kodaira dimension $\kod(S)>0$.
Let  $a\colon S\longrightarrow A$ be a strongly generating map to an abelian variety such that  $S$ is of maximal $a$-dimension.
Let $C_d$ be the nef divisor arising from the covering trick (\S \ref{covering-trick}).
Suppose that for $d\gg 0$  and for general $\alpha\in \Pic^0(A)$ the linear system $|C_d + \alpha|$ on $S_d$ is birational; then for $d\gg0$ the ratio  $C_d^2/h^0_a(S_d, C_d)$ tends to infinity with order at least quadratical in $d$.
\end{lem}

\begin{proof}
Let us fix $d$ and suppose that  for general $\alpha\in \Pic^0(A)$  the linear system $|C_d+ \alpha|$ on $S_d$ is birational.
Notice that by generic vanishing (\cite{gl}) we have that, for general  $\alpha\in \Pic^0(A)$ again,
$$
h^0_a(S_d,C_d)=h^0(S_d, C_d+ \alpha)= h^0(C_d, (C_d+ \alpha)_{|C_d}).
$$
Let us suppose that $C_d$ is a general  element in its linear system; the linear system $|C_d|$ is base point free, since it contains the linear system $a_d^*|H|$ and $H$ is very ample on $A$,  therefore $C_d$ is smooth and irreducible.
We want to apply Castelnuovo's bound to  the linear series defined by  $|(C_d+ \alpha)_{|C_d}|$, which induces a birational map:

We can compute  the  genus $g(C_d)$ of $C_d$ using the   adjunction formula and \eqref{cd}, \eqref{cdK}:
$$
g(C_d)= \frac{K_{S_d}\cdot C_d}{2}+\frac{C_d^2}{2}+1=d^{2q-2}\frac{K_S\cdot C}{2}+d^{2q-4}\frac{C^2}{2}+1.
$$
The bound in Theorem \ref{castelnuovo} is
$$
g(C_d)\leq \frac{m_d(m_d-1)}{2}(h^0_a(S_d,C_d)-2)+m_d\epsilon,
$$
 where $m_d =\left[\frac{C_d^2-1}{h^0_a(C_d)-2}\right]$.

 Let $M(d)$  be the ratio $C_d^2/h^0_a(S_d,C_d).$
 Observe that
 $$
 m_d =\left[\frac{C_d^2-1}{h^0_a(C_d)-2}\right]\leq \frac{C_d^2-1}{h^0_a(C_d)-2}\leq 3 M(d).
 $$
The inequality of Castelnuovo  and (\ref{cdK}) implies that
 $$
 d^{2q-2}\frac{K_S\cdot C}{2}+d^{2q-4}\frac{C^2}{2}+1\leq \frac{3}{2} M(d)^2 h^0_a(S_d,C_d) = \frac{3}{2} M(d) d^{2q-4}C^2.
 $$
Observe that  $K_S\cdot C>0$ as $S$ has  strictly positive Kodaira dimension, is of maximal $a$-dimension and $C$ is the pullback of a very ample divisor on $A$.
Hence we necessarily have that $M(d)$ grows at least as $d^2$.
\end{proof}

\bigskip

\section{Proof of the main result}
In this section we prove Theorem \ref{main}.
We  start with the ``only if'' part, assuming   that
$S$ is  a minimal surface of general type and of maximal Albanese dimension with $K^2_S=4\chi(S)$.

For $d>0$, let $C_d$ be the divisor given by the covering trick applied to $a=\alb_S$ (cf. \S \ref{covering-trick}).  
Let us fix a general $\alpha $ in $\Pic^0(A)$ and consider the following diagram:
\begin{equation}\label{fattorizzazione}
\xymatrix{
&&&\\
S_d  \ar[rr]^{(\alb_S)_d} \ar [d]_{f_d}& &A \ar@{^{(}->}[r]& \bP( H^0(A, H+\alpha)^\vee)\\
R_d\ar[urr]_{g_d}\ar[rr]_{\nu} && \overline{R_d}\ar[u]\ar@{^{(}->}[r]& \bP( H^0(S_d, C_d+\alpha)^\vee)\ar@{-->}[u]\\}
\end{equation}

Where:
\begin{itemize}
\item the surface $\overline{R_d}$ is the image of the map induced by  $|C_d+ \alpha|$ on $S_d$.
\item $\nu\colon R_d\longrightarrow \overline{R_d}$ is the normalization.
\end{itemize}
 Let us define $D_d:= g_d^*H$ on $R_d$.
 \begin{rem}\label{costruzione}
Let us make the following remarks:
\begin{itemize}
\item[(i)] The linear system $|D_d+ \alpha|=|g_d^*(H+ \alpha)|$ on $R_d$  is birational by construction, and $h^0(R_d, D_d+ \alpha)=h^0(S_d, C_d+ \alpha)$.
\item[(ii)] If $q(S)>2$ then the Kodaira dimension of $R_d$ is strictly positive, for any $d$, because $R_d$ dominates the image of the map $(\alb_S)_d$ which generates $A$ by construction.
\item[(iii)] By Lemma \ref{K-Cnef}, for $d\gg 0$ we have that both  $K_{S_d}-C_d$ and $C_d$ are nef. Since  $K^2_S=4\chi (S)$ then, by Lemma \ref{K-decomposable} we must have $C_d^2=4h^0_a(S_d,C_d).$
\item[(iv)] The map $f_d$ is generically finite, and it is  not birational for $d\gg0$ by (iii) and Lemma \ref{birational}.
\end{itemize}
\end{rem}

First we show that the map $f_d$ has degree 2.
Indeed, by Remark \ref{costruzione}(iii) for every $d\gg0$ we have the following chain of equalities
\begin{equation}\label{eq:D1}
(\deg f_{d})D_{d}^2=C_{d}^2=4h^0(S_{d},C_{d}+ \alpha)=4h^0(R_{d},D_{d}+ \alpha).
\end{equation}
Up to  passing to a  desingularization, we can suppose that $R_{d}$ is smooth; then
Theorem \ref{teomab} applied to the pair $(R_{d}, D_{d})$
 implies that
\begin{equation}\label{eq:D2}
D_{d}^2\geq 2 h^0_a(R_{d},D_{d})=2h^0(R_{d},D_{d}+ \alpha).
\end{equation}
Combining \eqref{eq:D1}, \eqref{eq:D2} and Remark \ref{costruzione} (iv), we obtain that $\deg f_d=2$ and   $$D_{d}^2=2h^0_a(R_{d},D_{d}+\alpha).$$

\smallskip

We now prove that the Kodaira dimension of $R_d$ is 0. Since $\kod(R_d)\ge 0$ and $D_d$ is nef and big, we have $K_{R_d}D_d\ge 0$ and $K_{R_d}D_d=0$ only if $\kod(R_d)=0$; hence the adjunction formula gives $$2r+2=D_d^2\le2g(D_d)-2,$$
 i.e.,  $r\leq g(D_d)-2$, with equality holding only if $\kod(R_d)=0$.

Consider now a general curve $D_d$, which is smooth since the system $|D_d|$ is free by construction, and for a general $\alpha \in {\rm Pic}^0(A)$ the linear system $L_{\alpha}:=|(D_d+\alpha)_{|D_d}|$. By the previous computation and generic vanishing on $R_d$, we have that $L_\alpha$ induces a $g^r_{2r+2}$ on $D_d$, where $r=h^0(R_d,D_d+\alpha)-1$.

Thus we have a  map $\Pic^0(A)\rightarrow W^r_{2r+2}(D_d)$ that is injective because the map $g_d$ is strongly generating, and $D_d^2>0$.

If $r \leq g(D_d)-3$  we can  apply an inequality  of Debarre and Fahlaoui (\cite{DF}, Proposition 3.3) on the dimension of abelian varieties contained in the Brill-Noether locus, that in our case  gives  $\dim A\le (d-2r)/2=1$, a contradiction.

So we have that $r=g(D_d)-2$ and $d=2g(D_d)-2$,  and therefore $\kod(R_d)=0$ by the previous remarks.

Hence   for $d\gg0$ the surface $R_d$ is birational to an abelian surface  by
 classification. Observe that the map $g_d$ is birational. Indeed, assume otherwise;  then  $g_d$  induces  a  non trivial \'etale map between the minimal model  $T_d$ of $R_d$ and the abelian surface $A$.  Since a map from a smooth surface to an abelian variety is a morphism,  the morphism $a_d\colon S_d\to A$ factorizes as $S_d\to T_d\to A$ and therefore $a_d$ is not strongly generating, contradicting   Lemma \ref{lem:Ctrick}. It follows that  for $d\gg 0$ the map $g_d$ is birational and
 ${\rm deg}a_d={\rm deg}f_d=2$.
\begin{rem}
It is worth remarking that a key point of our proof is that we are not taking the  limit  in the covering trick as in \cite{Pa}.
But it is  crucial that we are allowed to take $d\gg 0$, because the  properties proved above for the factorization (\ref{fattorizzazione}) can fail to hold for small $d$, as we now observe. 
For $d$ small, and  $H$ ample enough, it can well happen
that $C_d$ is  birational on $S_d$.
Consider for instance for the case $d=1$ the following situation: $\alb_S$ is a double covering, and $B$ is its ample branch locus. Then choosing $H=kB$, we have that
$(\alb_S)^*(H)=2 kK_S$, which is very ample on $S$ for $k\geq 3$.
In case $C_d$ is very ample, then the map $f_d$ is an isomorphism between $S_d$ and $R_d$ (in particular $\kod(R_d)=2$), and $g_d$ coincides with $(\alb_S)_d$.

Indeed, the process of taking $d$ big enough corresponds to choosing on $A$ a divisor which is  very ample but also ``comparatively small''.
\end{rem}

Up to  now we have proved that for a surface $S$ satisfying  the assumptions  of Theorem \ref{main}, the Albanese morphism
$$\alb_S\colon S\longrightarrow \Alb(S)$$
is a generically finite map of degree 2 onto the Albanese surface $\Alb(S)$.
We now see that the canonical model of $S$ is a double covering  of $\Alb(S)$, and describe its possible singularities.


Let us now consider $S$ as in Theorem \ref{main}, and consider the Stein factorization of $\alb_S$:
$$
\xymatrix{S \ar@/^1pc/[rr]^{\alb_S}\ar[r]_{\alpha}& S' \ar[r]_{\pi}& \Alb(S)},$$
where $S'$ is normal and $\pi$ is a double covering.
Let $B\subset A$ be the branch locus of $\pi$, which is reduced.
Note that the map $\alpha$ factors through the canonical model $S_{\rm can}$ of $S$. We are going to prove that indeed $S'=S_{\rm can}$.

Note that  as $S'$ is of general type, $B$ is ample; we can compute the  canonical invariants of $S'$ as follows:
\begin{equation}\label{ok}
K_{S'}^2=\frac{B^2}{2}; \quad \chi(S')=\frac{B^2}{8},
\end{equation}
so that $S'$ satisfies the Severi equality.

Let us perform the canonical resolution of the double covering $\pi$
(see \cite{bpv} III.7):
$$
\xymatrix@R-7pt{
 \widetilde S:=S_k \ar[r]^{\sigma_k} \ar@<+12pt>[d] & S_{k-1} \ar[r] \ar[d] & \cdots \ar[r]& S_1  \ar[r]^{\sigma_1} \ar[d] & S_0 =S' \ar@<-12pt>[d]^\pi \\
\widetilde A:= A_k \ar[r]^{\tau_k} & A_{k-1} \ar[r] & \cdots \ar[r]& A_1 \ar[r]^{\tau_1} & A_0 =A\\
}
$$
Recall that the maps $\tau_j$ are successive blow-ups that resolve the singularities of $B$;
the morphism $S_j\rightarrow A_j$
is the double covering with branch locus defined by the inductive formula:
$$B_j:=\tau_j^*B_{j-1}-2\left[\frac{m_{j-1}}{2}\right]E_j,$$ where $E_j$ is the exceptional
divisor of $\tau_j$, $m_{j-1}$ is the multiplicity for $B_{j-1}$ of the blown-up point, and $[\ \ ]$ denotes the integral part.
The surface $\widetilde S$ is smooth, not necessarily minimal, birational to $S$, and therefore $K^2_{\widetilde S}=K^2_S-\delta$, with $\delta\ge 0$.
One has the following relations (cf. \cite[Obs.1.16]{persson}):
$$
K_S^2 =K_{\widetilde S}^2+\delta
=K_{S'}^2-2\sum_{i=1}^{k}\left(\left[\frac{m_i}{2}\right]-1\right)^2 +\delta,
$$
and that
$$
\chi(S)=\chi(\widetilde S)=
\chi(S')-\frac{1}{2}\sum_{i=1}^{k}\left[\frac{m_i}{2}\right]\left(\left[\frac{m_i}{2}\right]-1\right).
$$
Recall now that a singularity  $P$ of the branch locus   $B$ on $A$ is called {\em negligible} (or  simple, inessential, etc\dots) if $K_S^2=K_{S'}^2$ and $\chi(S)=\chi(S')$.
By the above formulae (cf. \cite[Prop. 1.8]{persson}), this is equivalent to $P$ having multiplicity $\leq 3$ and all the points  infinitely near  to $P$ having multiplicity $\leq 2$.

So we have by assumption and by (\ref{ok})
$$
0=K_S^2-4\chi(S)= K_{\widetilde S}^2-4\chi(\widetilde S)+ \delta=
$$
$$
=K_{S'}^2-4\chi(S')+2\sum_{i=1}^{k}\left(\left[\frac{m_i}{2}\right]-1\right) +\delta=2\sum_{i=1}^{k}\left(\left[\frac{m_i}{2}\right]-1\right)+ \delta.
$$
In particular  we deduce that $\delta =0$ and $2\sum_{i=1}^{k}\left(\left[\frac{m_i}{2}\right]-1\right)=0$,
and the singularities of $B$ are necessarily negligible.

Now we need only to observe that, since the branch locus $B$   has negligible singularities, then the double cover $S'$ has Du Val singularities (cf. \cite[Table 1.18]{persson}); since there is a birational morphism   $S_{\rm can}\to S'$, we conclude that $S'=S_{\rm can}$.

On the other hand, it is clear from the above discussion that a double cover of an abelian surface with ample branch locus $B$ with negligible  singularities is the canonical model of a general type surface of maximal Albanese dimension lying on the Severi line.

     \end{document}